\documentclass[10pt]{amsart}
\usepackage{amsfonts}
\usepackage{amssymb}
\usepackage{euscript}
\newtheorem{theorem}{Theorem}
\newtheorem{lemma}{Lemma}
\newtheorem{prop}{Proposition}

\newtheorem{cor}{Corollary}
\newtheorem{example}{Example}
\def\SO{\operatorname{SO}}
\def\GL{\operatorname{GL}}
\def\SL{\operatorname{SL}}
\def\PSO{\operatorname{PSO}}
\def\Lie{\operatorname{Lie}}
\def\Mat{\operatorname{Mat}}
\def\mod{\operatorname{mod}}
\def\C{\mathbb C}

\def\P{\mathbb P}
\def\g{\mathfrak g}
\def\gl{\mathfrak {gl}}
\def\Z{\mathbb Z}
\def\Gan{\mathbb G_a^n}
\def\exp{\operatorname{exp}}
\def\PGL{\operatorname{PGL}}

\def\Aut{\operatorname{Aut}}
\def\Spec{\operatorname{Spec}}
\def\Pic{\operatorname{Pic}}
\def\so{\mathfrak{so}}
\def\sl{\mathfrak{sl}}
\begin{document}
\title[Automorphisms of the quadric]{Hassett-Tschinkel correspondence and automorphisms of the quadric}
\author{Elena Sharoyko}

\begin{abstract} We study locally transitive actions of the commutative unipotent group $\Gan$ on a nondegenerate quadric in the projective space $\P^{n+1}$. It is shown that for each $n$ such an action is unique up to isomorphism.
\end{abstract}
\maketitle

\begin{section}{Introduction}

This work is devoted to locally transitive actions (i.e. actions with an open orbit) of the commutative unipotent algebraic group $\Gan=\underbrace{\mathbb G_a\times\mathbb G_a\times\ldots\times\mathbb G_a}_{n \mbox { times }}$  on complex projective varieties. Here $\mathbb G_a=(\C,+)$. B.\:Hassett and Yu.\:Tschinkel \cite{H-T} described a remarkable correspondence between locally transitive actions of the group $\Gan$ on the projective space $\P^n$ and complex local commutative algebras of dimension $n+1$. From this construction and classification results in \cite{S-T} it follows that such locally transitive actions form a finite number of equivalence classes if and only if $n\le ~5$. The problem of describing of locally transitive $\Gan$-actions on a nondegenerate quadric $Q_n\subset\P^{n+1}$ was stated in \cite[Question 3-(4)]{H-T}. We show that  for each $n$ there exists a unique locally transitive action of the group $\Gan$ on the quadric $Q_n$ up to isomorphism. To prove this, Hassett-Tschinkel correspondence should be generalized to the actions on quadrics. This way leads us to complex local commutative algebras of dimension $n+2$ with some special features: there is a certain nondegenerate quadratic form on such algebra and the square of the maximal ideal in this algebra is one-dimensional. Let us note that this result cannot be extended to other smooth projective hypersurfaces. If $H_{n,d}$ is a smooth projective hypersurface of degree $d\ge 3$ in $\P^{n+1}$, then the group of its automorphisms is finite except for the case $n=2$, $d=4$ \cite[Thm.~2]{M-M}. In the latter case the automorphism group of the hypersurface is discrete \cite[Thm.~4]{M-M}.     
 
The action we construct also determines the locally transitive $\Gan$-action on the affine quadric. The automorphism groups of affine quadrics are discussed in various works, cf. \cite{D} и \cite{Tot}. 

It is shown in \cite{H-T} that for each $n$ there exists a unique $\Gan$-action on $\P^n$ with finite number of orbits.
It seems interesting to find out the modality of another actions. Here the term modality is used for the maximal amount of parameters in a continuous family of $\Gan$-orbits. In Section 5 there is an example of such calculation in a special case. We will also prove that the modality of the locally transitive action of the group $\Gan$ on $Q_n$ is equal to $n-2$ for $n\ge 2$.  Some results on maximal value of modality over projective (resp., affine) embeddings of a homogeneous space of a reductive group can be found in   \cite{Ah} (resp., \cite{Ar}). There is no corresponding result for projective embeddings of  unipotent groups. 

Each locally transitive $\Gan$-action on $\P^n$ (resp., on $Q_n$) determines a maximal commutative unipotent subgroup of dimension $n$ in $\SL(n+1)$ (resp., $\SO(n+2)$).
The last section contains an example of a maximal commutative unipotent subgroup of dimension $n$ in $\SO(n+2)$ with non-locally transitive action associated to it. We do not have similar examples for the group $\SL(n+1)$.

The author is grateful to I. V. Arzhantsev for posing of the problem and useful discussions, and also to Yu. G. Prokhorov for important comments.

\end{section}
\begin{section}{Hassett-Tschinkel correspondence}

This section briefly recalls the results of \cite{H-T} that establish a correspondence between cyclic representations of the group $\Gan$ and finite-dimensional local algebras.

Let $\rho:\Gan\to\GL_l(\C)$ be a rational representation and $v\in \C^l$ be a fixed cyclic vector, i.e.  $\langle\rho(\Gan)v\rangle=\C^l$. It determines the representation $d\rho:\g\to\gl_l(\C)$  of the Lie algebra $\g$ of the group $\Gan$, which induces the representation $\tau: U(\g)\to\Mat_{l\times l}(\C)$ of the universal enveloping algebra $U(\g)$.  Since the group $\Gan$ is commutative, the algebra $U(\g)$ is naturally isomorphic to the polynomial algebra $\C[S_1,\ldots, S_n]$, where $\g\cong\langle S_1,\ldots, S_n\rangle.$ The subspace $\tau(U(\g))v$ is $\g$- and $\Gan$-invariant; since it contains $v$, we can see that $\tau(U(\g))v=\C^l.$ Let us consider the set $I=\{a\in U(\g):\tau(a)v=0\}$. It is an ideal of $U(\g)$ and $U(\g)/I\cong\tau(U(\g))v=\C^l$. Using this quotation, to each transformation $\tau(a)$ of the space $\C^l$ one can assign  the left multiplication by the element $a\in U(\g)$ in $R=U(\g)/I$. Moreover, the vector $v$ is associated with the class of 1  in the algebra $R$. Since all the operators $\tau(S_1),\ldots, \tau(S_n)$ are nilpotent, the ideal $I$ contains some powers of each generator $S_1,\ldots, S_n$. So the ring $R$ is local with a unique maximal ideal $\mathfrak m_R=\langle S_1,\ldots, S_n\rangle$ and $\dim_\C R=l$.
Conversely, let us associate with $R=\C[S_1,\ldots, S_n]/I$ an $l$-dimensional representation $\rho$ of the group $\Gan$. Here the operator $\rho((\alpha_1, \ldots, \alpha_n))$ is a left multiplication by the element $\exp(\alpha_1 S_1+\ldots+\alpha_nS_n)$ in $R$, and the vector $1\in R$ is the marked one. It is easy to see that the representation $\rho$ is effective if and only if $\langle S_1,\ldots, S_n\rangle\cap I=0.$ In this case we say that the ring $R$ is  {\it nondegenerate}.

Now proceed to an action of the group $\Gan$ in the polynomial ring $\C[x_1,\ldots, x_n]$, where an element $(\alpha_1,\ldots, \alpha_n)$ brings $x_i$ to $x_i+\alpha_i$. Here the basic elements $S_1,\ldots, S_n$ of the algebra $\g$ correspond to the operators $\frac{\partial}{\partial x_1},\ldots, \frac{\partial}{\partial x_n}. $ Consider the subspace 
$$V=\{f\in\C[x_1,\ldots, x_n]:F(f)=0 \quad\forall F\in I\},$$
where by $F(S_1,\ldots, S_n)$ we denote the differential operator $F(\frac{\partial}{\partial x_1},\ldots, \frac{\partial}{\partial x_n})$. From the proof of \cite[Prop. 2.11]{H-T} it follows that the dimension of the space $V$ is equal to $l$ and the induced representation  of the group $\Gan$ on $V$ is dual to $\rho$. Conversely,  to each $\Gan$-invariant subspace $V\subset \C[x_1,\ldots, x_n]$ of dimension $l$ one can assign an ideal  
$I$ in $\C[S_1,\ldots, S_n]$, namely  $I=\{F:F(f)=0\quad \forall f\in V\}$, and a representation $\rho$ of the group $\Gan$ in the space $R=\C[S_1,\ldots, S_n]/I$. The structure of this representation was described before.

The algebra $R$ is nondegenerate if and only if  no differential operator of first order  annihilates the space $V$. This condition is also equivalent to the following property: the transcendency degree of a subalgebra in $\C[x_1,\ldots, x_n]$ generated by $V$ is equal to $n$. We call such subspaces  $V\subset \C[x_1,\ldots, x_n]$  {\it nondegenerate}. The following result can be obtained in this way:

\begin{theorem}\label{HTosn}\cite[Th. 2.14]{H-T} There is a one-to-one correspondence among the following:
\begin{enumerate}
\item isomorphism classes of pairs $(\rho, v)$ such that $\rho:\Gan\to\GL_l(\C)$ is an effective rational representation and $v\in\C^l$ is a cyclic vector;
\item nondegenerate subschemes $\Spec R \subset\Spec\C[S_1,\ldots, S_n]$ of length $l$ supported at the origin;
\item classes of nondegenerate subspaces $V\subset\C[x_1,\ldots, x_n]$ of dimension $l$, which are invariant under the isomorphisms  $x_i\mapsto x_i+\alpha_i$, $\alpha_i\in\C$, of the ring $\C[x_1,\ldots, x_n]$; the equivalence is  induced by linear transformations of the space $\langle x_1,\ldots, x_n\rangle$.
\end{enumerate} 
\end{theorem}

An effective linear representation $\rho:\Gan\to\GL_l(\C)$ with a marked cyclic vector $v\in\C^l$ determines an action of the group $\Gan$ on the space $\P^{l-1}$ with an open orbit $\Gan\langle v\rangle$ if and only if $n=l-1.$ In this case the images of the generators $S_1,\ldots, S_n$ form a $\C$-basis in the maximal ideal $\mathfrak m_R$ of the ring $R=\C[S_1,\ldots, S_n]/I$.  So in this case there is no need to fix the realisation of $R$ as a quotient-ring.

\begin{theorem}\label{HTloc}\cite[Prop. 2.15]{H-T} The following are equivalent:
\begin{enumerate}
\item equivalence classes of locally transitive actions of the group $\Gan$ on $\P^n$;
\item local algebras of $\C$-dimension $n+1$, up to isomorphism.
\end{enumerate}
\end{theorem}

Let us note that any algebraic homomorphism $\rho:\Gan\to\PGL_l(\C)$ can be uniquely extended to a rational representation $\tilde\rho:\Gan\to\GL_l(\C)$. This explains the implication $(1)\Rightarrow (2)$.

All local finite-dimensional algebras of degree $\le 5$ are listed in \cite{H-T}. Let us mention an inaccuracy in this list. Proposition 3.4 in \cite{H-T} says that there are ten local  algebras of dimension 5 and specifies them all up to isomorphism:
\medskip
\newline $R_1=\C[S_1,S_2,S_3, S_4]/(S_1^2-S_2, S_1S_2-S_3,S_1S_3-S_4, S_2S_3, S_1S_4)$, 
\newline $R_2=\C[S_1,S_2,S_3, S_4]/(S_1^2-S_3, S_1S_2, S_2^2, S_1S_3-S_4, S_1S_4)$,
\newline $R_3=\C[S_1,S_2,S_3, S_4]/(S_1^2-S_3, S_1S_2, S_2^2-S_3, S_1S_3-S_4, S_1S_4)$, 
\newline $R_4=\C[S_1,S_2,S_3, S_4]/(S_1^2-S_3, S_1S_2-S_3, S_2^2, S_1S_3-S_4, S_1S_4)$, 
\newline $R_5=\C[S_1,S_2,S_3, S_4]/(S_1^2-S_3, S_1S_2-S_4, S_2^2, S_1S_3, S_2S_3, S_1S_4)$, 
\newline $R_6=\C[S_1,S_2,S_3, S_4]/(S_1^2-S_3, S_2^2-S_4, S_1S_2, S_1S_3, S_2S_4)$, 
\newline $R_7=\C[S_1,S_2,S_3, S_4]/(S_1^2-S_4, S_1S_2,S_1S_3, S_2S_3-S_4, S_2^2, S_3^2, S_1S_4)$, 
\newline $R_8=\C[S_1,S_2,S_3, S_4]/(S_1^2, S_2^2, S_3^2, S_1S_2-S_4,S_1S_3, S_2S_3)$, 
\newline $R_9=\C[S_1,S_2,S_3, S_4]/(S_1^2-S_4, S_2^2, S_3^2, S_1S_2, S_1S_3, S_2S_3, S_1S_4)$, 
\newline $R_{10}=\C[S_1,S_2,S_3, S_4]/(S_iS_j | i,j=1,\ldots, 4)$.
\medskip
\newline By $I_j$ we denote an ideal corresponding to the algebra $R_j$. Then it is easy to see that both ideals $I_3$ and $I_4$ contain $S_4$.  Indeed, in $R_3$ we have $S_4=S_1S_3=S_1S_2^2=S_2(S_1S_2)=0$, and in $R_4$ one can see that $S_4=S_1S_3=S_1^2S_2=S_2S_3=S_1S_2^2=0$. It means that the dimension of $R_3$ and $R_4$ cannot exceed 4. Now let us find out which algebras satisfy the conditions $\dim \mathfrak m_R=4$, $\dim \mathfrak m_R^2=2, \dim \mathfrak m_R^3=1$ (\cite{H-T} claims that these are $R_2, R_3$ and $R_4$).

\begin{lemma}\label{error} There are two local algebras of dimension 5 satisfying the conditions $\dim \mathfrak m_R=4$, $\dim \mathfrak m_R^2=2$, $\dim \mathfrak m_R^3=1$. They can be written as follows: 
$$R_2=\C[S_1,S_2,S_3, S_4]/(S_1^2-S_3, S_1S_2, S_2^2, S_1S_3-S_4, S_1S_4)$$
and
$$R_3'=\C[S_1,S_2,S_3, S_4]/(S_1^2-S_3, S_1S_2, S_2^2-S_4, S_1S_3-S_4, S_2S_3, S_1S_4).$$
\end{lemma}

\begin{proof}
Suppose that $\mathfrak m_R^3=\langle S_4\rangle$ and $\mathfrak m_R^2=\langle S_3, S_4\rangle$. Then $S_iS_4=0$ for all $i$ and $S_3^2=0$. We need to determine the products of other basic vectors. Let us write these products as follows:
$$S_1^2=aS_3+bS_4, S_1S_2=cS_3+dS_4, S_2^2=eS_3+fS_4, S_1S_3=gS_4, S_2S_3=hS_4.$$
Without loss of generality one can put $g\ne 0$, because $S_3\in \mathfrak m_R^2$ and the subset $\{S\in\langle S_1, S_2\rangle: SS_3\ne 0 \}$ is dense in $\langle S_1, S_2\rangle$. Consider three cases.
\smallskip

a) Suppose that $h=0, f\ne 0$. Then we have
$$0=(aS_3+bS_4)S_2=S_1^2S_2=S_1(cS_3+dS_4)=cgS_4,$$
which means that $c=0$. In a similar way, one can see that
$$0=(cS_3+dS_4)S_2=S_1S_2^2=S_1(eS_3+fS_4)=egS_1,$$
hence $e=0$. So the relations in our algebra can be simplified in the following way:
$$S_1^2=aS_3+bS_4, S_1S_2=dS_4, S_2^2=fS_4, S_1S_3=gS_4, S_2S_3=0.$$
The square of the ideal is two-dimensional, hence $a\ne 0$. 
Let us replace $S_1$ by $\gamma S_1+\delta S_2$ such that $\gamma d+\delta f=0$; then $S_1S_2=0$. One can choose $\gamma\ne 0$, since $f\ne 0$. Then let us replace $S_3$ by the linear combination of elements  $S_3$ and $S_4$ such that $S_1^2=S_3$. Now the relations can be rewritten once more:
$$S_1^2=S_3, S_1S_2=0, S_2^2=fS_4, S_1S_3=\gamma^3agS_4, S_2S_3=0.$$
We can get $\gamma^3ag=1$ and $f=0$ or $f=1$ by multiplying $S_i$ by appropriate scalars. In these cases one obtains two algebras described earlier. 
\smallskip

b) Suppose that $h=f=0$. As in a), we have $c=e=0$ and  the relations can be written in the following way:
$$S_1^2=aS_3+bS_4, S_1S_2=dS_4, S_2^2=0, S_1S_3=gS_4, S_2S_3=0.$$
Let us replace $S_1$ by $kS_1$, $S_2$ by $gS_2-dS_3$, $S_3$ by $k^2aS_3+k^2bS_4$, where $agk^3=1$. Then 
$$S_1^2=S_3, S_1S_2=0, S_2^2=0, S_1S_3=S_4, S_2S_3=0,$$
hence the algebra is isomorphic to $R_2$.
\smallskip

c) Suppose $h\ne 0$. Let us replace $S_2$ by $\kappa S_1+\lambda S_2$, where $\kappa g+\lambda h=0$ and $\lambda\ne 0$. Then $S_2S_3=0$ and we come to the situation a).

It is also necessary to check that the algebras $R_2$ and $R_3'$ are not isomorphic to each other. Note that there is  an element  $S$ in $R_3'$ such that $S^2\in\mathfrak m_R^3\setminus\{0\}$ and there is no such element in $R_2$.
\end{proof} 

It is easy to check that the other algebras from the list are of dimension 5. In each case we can  construct a subalgebra in $\Mat_{5\times 5}(\C)$ of dimension 5 with specified relations among its generators.

\end{section}

\begin{section}{Local finite-dimensional algebras and actions on the quadric}

In this section some useful facts will be proved. It will let us reformulate our problem and use methods of Hassett and Tschinkel. It is important for listing local finite-dimensional algebras, which correspond to locally transitive actions on the quadric.

Begin with the following well-known result. We shall prove it as we do not have an accurate reference. 

\begin{lemma}\label{aut}
$\Aut Q_n=\PSO(n+2)$.
\end{lemma}

\begin{proof} Let us note that if $n\ge 3$, then the Picard group of the quadric $\Pic Q_n\cong\Z$ is generated by the line bundle $\mathcal O(1)$. Any automorphism of the quadric $Q_n$ induces the automorphism of its Picard group. This map could bring the generator $\mathcal O(1)$ either to $\mathcal O(1)$ or to $\mathcal O(-1)$, but the last line bundle does not have global sections. So any hyperplace section of the quadric comes again to a hyperplane section. The last statement is also true if $n=1$ or $n=2$. In this way to each automorphism of the quadric  we assign a transformation of the space  $\P(H^0(\mathcal O(1)))=(\P^{n+1})^*$ of hyperplane sections of the quadric. The conjugate transformation of $\P^{n+1}$ extends the original automorphism on this space.
\end{proof}

\begin{cor} Any action of the group $\Gan$ on the quadric $Q_n\subset\P^{n+1}$ can be extended to an orthogonal representation of the group $\Gan$ in the space $\C^{n+2}$.
\end{cor}

Let us remark that the similar problem of extending the action on an affine quadric is extremely hard, see \cite{D}.

We say that {\it $n$-factor} is a factoralgebra of the polynomial algebra $\C[S_1,\ldots, S_n]$ with linearly independent images of generators $S_1,\ldots, S_n$. Let us note that an algebra can be represented as an $n$-factor if it is generated by $n$ linearly independent elements. Consider a pair $(\rho,v)$, where $\rho$ is a faithful representation of the group $\Gan$ in the space $V=\C^{n+2}$ and $v$ is a cyclic vector in this representation. To this pair assign a local algebra $R$ of dimension $n+2$, using the  Hassett-Tschinkel correspondence. This algebra is an $n$-factor. The images of the generators $S_i$ in $R$ are linearly independent as far as the representation $\rho$ is faithful. 

The $\C$-basis $\{\mu_1,\ldots,\mu_{n+2}\}$ of an $n$-factor $R$ is called {\it proper} if  $\mu_1=1$, the element $\mu_{n+2}$ can be expressed as a polynomial of other basic vectors and any power of the maximal ideal $\mathfrak m_R$ is spanned by some of those basic vectors. In the sequel we shall need basis of this kind.  

\begin{lemma}\label{factor}
Any local algebra of dimension $n+2$ supported at the origin  except for $B_n=\C[S_1,\ldots, S_{n+1}]/(S_iS_j | i,j=1,\ldots, n+1)$ is  an $n$-factor.
\end{lemma}

\begin{proof}
Suppose an algebra $R$ of dimension $n+2$ is not an $n$-factor. Then the product of any two elements in the maximal ideal is a linear combination of the factors. Indeed, if not, then the elements   $1, f_1, f_2, f_1f_2$ (or $1, f_1, f_1^2$) are linearly independent. Let us add some more vectors to obtain a basis of the algebra $R$. Then the element $f_1f_2$ (or $f_1^2$ respectively) can be expressed as a polynomial of other basic elements, hence the algebra is an $n$-factor. Consider three cases: $f_1f_2=f_1+f_2$ (maybe we have to replace $f_1$ and $f_2$ by their multiples), $f_1f_2=f_1$ and $f_1f_2=0$. In the first two cases any power of the maximal ideal has to contain the element $f_1+f_2$ or $f_1$ respectively, which contradicts to the description of our algebra. So the product of any two elements in the maximal ideal is zero, which proves the statement. 
\end{proof}

\begin{lemma}\label{ort}
Suppose a local algebra $R$ of dimension $n+2$ is an $n$-factor. 
The closure of the orbit $\Gan\langle 1\rangle\subset \P(R)$ is a nondegenerate quadric if and only if there exists a nondegenerate quadratic form $q$ on $R$ such that the representation $\rho$ is orthogonal and the vector $1$ is isotropic with respect to this form.
\end{lemma}

\begin{proof} 
{\it Necessity.} If a projective quadric coincides with the orbit-closure, then it is $\Gan$-invariant. So the cone over this quadric is also $\Gan$-invariant. This cone  is given by the same equation $q(u)=0$ in $R$. From \cite[Th. 3.1]{t55} it is known that the zero locus of a polynomial $f$ is $G$-invariant if and only if $f$ is $G$-semiinvariant. In our case the group $\Gan$ is unipotent, so the quadratic form $q(u)$ has to be $\Gan$-invariant. It means exactly the ortogonality of the representation. Moreover, the vector 1 is in the zero locus of this polynomial i.e. it is an isotropic vector.  

{\it Sufficience.} Suppose we have an orthogonal representation and an isotropic vector 1. Then the zero locus of the polynomial $q(u)$ is invariant and it contains 1. Hence the zero locus of the same quadratic form in projective space is also invariant. The representation is faithful, so the orbit of the point $1$ and the group $\Gan$  are of the same dimension. Therefore, the orbit of the point  $\langle 1\rangle$ is $n$-dimensional, since our group is unipotent. So the orbit-closure and the projectivisation of the quadric $q(u)=0$ are of the same dimension and both irreducible, hence they coincide.
\end{proof}

We take the images of the generators $S_i$ of the algebra $\C[S_1,\ldots, S_n]$ as basic vectors $\mu_i$, $i=2, \ldots, n+1$. It is always possible to change the set $\{S_i\}$ so that the corresponding basis in $R$ would be proper.

\begin{lemma}\label{ortmu}
The representation $\rho$ is orthogonal if and only if any operator of multiplication by $\mu_i$ in $R$ annihilates the bilinear form $B$ corresponding to the quadratic form $q$, i. e. $B(\mu_ix, y)+B(x, \mu_iy)=0$ as  $i=2, \ldots, n+1$ , $x, y\in R$.
\end{lemma}

\begin{proof}
The representation $\rho$ is orthogonal if and only if all operators of multiplication by $\mu_i$ are antisymmetric with respect to the form $B$, since these operators are images of the elements $(0,\ldots,0,1,0,\ldots, 0)\in\Lie\Gan$ 	under the differential of the representation $\rho$.
\end{proof}
\end{section}

\begin{section}{The main theorem}
Let $R$ be an $n$-factor of dimension $n+2$ with a proper basis of $\{\mu_i\}$. If $\dim \mathfrak m_R^2=1$ and $\mathfrak m_R^2=\langle \mu_{n+2}\rangle$, then one can construct the following bilinear map $B_0$ in the space $R$: 
\newline $B_0(\mu_1,\mu_i)=-\delta_{i,n+2}$; 
\newline $B_0(\mu_i,\mu_{n+2})=-\delta_{i1}$; 
\newline$B_0(\mu_i,\mu_j)=a_{ij}$, where $\mu_i\mu_j=a_{ij}\mu_{n+2}$ as $i,j=2,\ldots, n+1$ and $\delta_{ij}$ is the Kronecker symbol.  

\begin{theorem}\label{usl}
The orbit-closure of the line $\langle v\rangle\in \P^{n+1}$ spanned by the cyclic vector $v$ is a nondegenerate quadric if and only if the corresponding algebra $R$ is such that $\dim \mathfrak m_R^2=1$ and the form $B_0$ is nondegenerate. In this case the orbit-closure is given by the equation $B_0(u,u)=0$.
\end{theorem}

\begin{proof}
Assign a finite-dimensional algebra $R$ to the pair $(\rho,v)$. The vector $v$ becomes a unity in this algebra.  We are to find out when it is possible to equip the algebra $R$ with a nondegenerate $\Gan$-invariant bilinear form $B$ such that the vector 1 is isotropic. Suppose that there is such a form and let us study its properties.

\begin{lemma}\label{svoj} Suppose $R$ is an $n$-factor equipped with a proper basis and a nondegenerate invariant bilinear form $B$ such that the vector 1 is isotropic. Then the following conditions are fulfilled:
\begin{itemize}
\item [1.] $B(1,\mu_i)=0$ as $i=2,\ldots, n+1$;
\item [2.] $B(\mu_i,\mu_j)=-B(1,\mu_i\mu_j)$ as $i,j=2,\ldots,n+1$;
\item [3.] $B(\mu_i,\mu_{n+2})=0$ as $i=2,\ldots, n+2$;
\item [4.] $B(x,y)= B_0(x,y)$, if the form is normalized such that $B(1,\mu_{n+2})=-1$;
\item [5.] $\dim \mathfrak m_R^2$=1;
\end{itemize}

\end{lemma}

\begin{proof}
1) We know from Lemma 5 that the operator of multiplication by $\mu_i$ ($i=2,\ldots,n+1$) annihilates any $\Gan$-invariant bilinear form  $B$. So for $i=2,\ldots, n+1$ one can see that
$$B(1, \mu_i)+B(\mu_i,1)=0.\eqno(1)$$

2) The next condition
$$B(\mu_i,\mu_j)=-B(1,\mu_i\mu_j)\eqno(2)$$
can be proved in a similar way for all $i,j=2,\ldots, n+1$.

From (1) it follows that the vector $\mu_1$ is orthogonal to each of the vectors $\mu_2,\ldots,\mu_{n+1}$ and to itself also (since the vector 1 is isotropic). The form $B$ is nondegenerate, so  $B(1,\mu_{n+2})\ne 0.$  Hence one can put $B(1,\mu_{n+2})=-1.$ From (2) it can be easily deduced that $B(\mu_i,\mu_j)$ ($i,j=2,\ldots,n+1$) equals to the coefficient by $\mu_{n+2}$ in the decomposition of the product $\mu_i\mu_j$ in our basis. Therefore we need only to define $B(\mu_{n+2},\mu_i)$ whereas $i=2,\ldots,n+2$. 

Let us show that $\mu_{n+2}$ is in the  socle $z(R)$ of the ideal $\mathfrak m_R$, which is the maximal nonzero power of this ideal. Indeed, suppose the opposite; then take any element $\mu=\alpha_2\mu_2+\ldots+\alpha_{n+1}\mu_{n+1}\in z(R)$ (there is no $\alpha_{n+2}$, since the basis is proper). In this case for  $i=2, \ldots, n+2$ we have $B(\mu_i,\mu)=-B(\mu\mu_i,1)=0$, since the product $\mu\mu_i$ should belong to the ideal $z(R)\mathfrak m_R=0$. Moreover, one can see that  $B(1,\mu)=\sum_{i=2}^{n+1}\alpha_i B(1,\mu_i)=0$.  Hence the element  $\mu$ is in the kernel of the form and so it is equal to zero. The contradiction proves our statement. 

5) Note that
$$B(\mu_i,\mu_j\mu_k)=-B(\mu_i\mu_j,\mu_k)=B(\mu_j,\mu_i\mu_k)=-B(\mu_j\mu_k,\mu_i),$$
hence for any любых $i,j, k=2,\ldots,n+1$ we have
$$B(\mu_i,\mu_j\mu_k)=0.\eqno (3)$$ 
It is also true if $i$, $j$ or $k$ is equal to $n+2$, since the element $\mu_{n+2}$ is in the socle of the ideal.  
Therefore, any element of the ideal $\mathfrak m_R^2$ is orthogonal to $\langle \mu_2,\ldots,\mu_{n+2}\rangle$, so this square should lie in $\langle \mu_{n+2}\rangle$. That means that the dimension of this square cannot exceed 1. If it is zero then $R$ is not an $n$-factor. It proves the statement 5.

Now we can finish with the statements 3 and 4 by defining of the form $B$ at last. One has $B(\mu_i,\mu_{n+2})=-B(1,\mu_i\mu_{n+2})=-B(1,0)=0$, since the element $\mu_{n+2}$ is in the socle. It is also easy to show that  $B(\mu_{n+2},\mu_{n+2})=0$  by representing the element $\mu$ as a polynomial of other $\mu_i$'s. So the vector $\mu_{n+2}$ is orthogonal to the whole ideal $\mathfrak m_R$ and is not orthogonal to the vector 1. The form $B$ is symmetric, hence $B(\mu_{n+2},1)=-1$. Our form is therefore uniquely defined up to multiplication by a scalar.
\end{proof}

Now half of Theorem \ref{usl} follows from Lemma~\ref{svoj}. Conversely, suppose that the ideal $\mathfrak m_R^2$ is one-dimensional and the corresponding form is nondegenerate. By definition of the form $B_0$ we have $B_0(\mu_i\mu_j,\mu_k)=0$ for all $i,j,k=2,\ldots,n+2$, so this form is invariant with respect to the action of the group. Moreover, in this case the algebra $R$ is an $n$-factor, and it follows from Lemmas in Section 3 that the orbit-closure of the line spanned by a cyclic vector (or, equivalently, the line $\langle 1\rangle$ in the corresponding finite-dimensional algebra) is the quadric $B_0(u,u)=0$. This concludes the proof.
\end{proof}

\begin{theorem}\label{unit} For each $n$ there exists a unique (up to isomorphism) pair $(\tilde\rho, v)$, where $\tilde\rho$ is a projective representation of the group $\Gan$ and $v$ is a cyclic vector in the corresponding space, such that the orbit-closure of the line $\langle v \rangle$ is a nondegenerate quadric in $\P^{n+1}$.
\end{theorem}

\begin{proof}
From the previous theorem it follows that the quadric $Q_n$ is an orbit-closure if and only if in the algebra $R$ the square of the maximal ideal is one-dimensional and the form $B$ is nondegenerate. In this case the matrix of the form $B$ is the following:
$$\left(\begin{array}{ccccc}0&0&\ldots&0&-1\\\vdots&&A&&\vdots\\-1&0&\ldots&0&0\end{array}\right).$$
One can change the basis in $\langle\mu_2,\ldots, \mu_{n+1}\rangle $ to replace this matrix by the matrix
$$\left(\begin{array}{ccccc}0&0&\ldots&0&-1\\\vdots&&E_n&&\vdots\\-1&0&\ldots&0&0\end{array}\right),$$
where $E_n$ is the identity matrix  (it is possible because the restriction of the form $B$ on the subspace $\langle \mu_2,\ldots, \mu_{n+1}\rangle$ is nondegenerate).
We need also to note that if the square of the maximal ideal is one-dimensional, then the form $R$ and the multiplication in the algebra $R$ determine each other. So any algebra which corresponds to the action of this type has to be isomorphic to the algebra $R=\C[S_1,\ldots, S_n]/(S_k^2-S_l^2, S_iS_j\mid k,l=1,\ldots, n, i\ne j)$.
\end{proof}

{\bf Remark 1.} Consider any local algebra of dimension $n+2$ such that the square of its maximal ideal is one-dimensional. Then we can assign to it a quadric as an orbit-closure of the line $\langle 1\rangle$. However this quadric can be degenerate. Moreover, a quadric can be an orbit-closure for actions of another type. So in general a quadric does not determine an algebra uniquely. As an example one can take a quadric $y_1^2-2y_0y_2=0$ in $\P^3$ which corresponds both to algebra $A_1=\C[x,y]/(x^3-y, x^4)$ and to $A_2=\C[x,y]/(y^2, xy, x^3)$.  The corresponding actions can be represented by the matrices
$$\rho_1(a)=\left(\begin{array}{cccc}1&0&0&0\\
a_1&1&0&0\\\frac{a_1^2}{2}&a_1&1&0\\a_2+\frac{a_1^3}{6}&\frac{a_1^2}{2}&a_1&1\end{array}\right)$$
and
$$\rho_2(a)=\left(\begin{array}{cccc}1&0&0&0\\
a_1&1&0&0\\\frac{a_1^2}{2}&a_1&1&0\\a_2&0&0&1\end{array}\right),$$
where $a=(a_1,a_2)\in\mathbb G_a^2$. In the former case the orbit-closure of the open orbit contains three orbits while in the latter case there is an infinite amount of orbits in the orbit-closure.

\end{section}

\begin{section}{Modality of locally transitive actions}

This section is devoted to more detailed study of locally transitive actions of the group $\Gan$. In particular, we will discuss when the orbit-closure of the line spanned by a cyclic vector contains a finite amount of orbits.

The paper \cite{H-T} contains the following proposition:
\begin{prop}\label{finite}\cite[Prop. 3.7]{H-T}
For each $n$ there exists a unique locally transitive action of the group $\Gan$ on the space $\P^n$ with a finite number of orbits. This action corresponds to the local algebra $R=\C[S_1,\ldots, S_n]/(S_1^i-S_i, S_jS_k\mid i=1,\ldots,n, j+k>n)$.
\end{prop}

Suppose that an algebraic group $H$ acts on an irreducible variety $X$. By $d(X, H)$ we denote the minimal codimension of an $H$-orbit in $X$.  The {\it modality} $\mod(Y, H)$ of an action $H:Y$ equals to maximum of the value $d(X,H)$ over irreducible invariant subvarieties  $X\subset Y$. In particular,  $\mod(Y,H)=0$ if and only if the number of $H$-orbits in $Y$ is finite.

Let us pose a natural question to generalize Proposition \ref{finite}: what is the possible modality of a locally transitive action of the group $\Gan$ on $\P^n$? 

\begin{example} Suppose that an algebra $R$ is of the form $\C[x,y]/(x,y)^{k+1}$. We shall prove that in this case the modality of the corresponding action on the projective space equals to $pk-p^2+p$, where $p=[\frac{k+1}{2}]$.
\end{example}

Begin with a special case $k=3.$ We need the matrix $\rho(a)$ ($a=(a_1,\ldots,a_9)\in\mathbb G_a^9$) of the corresponding representation. So let us calculate an exponent of an element $a_1\mu_1+\ldots+a_9\mu_9=a_1x+a_2y+a_3x^2+\ldots+a_9y^3$ and  multiply it by different  $\mu_i$. The following matrix is obtained:
$$\left(\begin{array}{cccccccccc}1&0&0&0&0&0&0&0&0&0\\
a_1&1&0&0&0&0&0&0&0&0\\
a_2&0&1&0&0&0&0&0&0&0\\
a_3+a_1^2/2&a_1&0&1&0&0&0&0&0&0\\
a_4+a_1a_2&a_2&a_1&0&1&0&0&0&0&0\\
a_5+a_2^2/2&0&a_2&0&0&1&0&0&0&0\\
a_6+a_1a_3+a_1^3/6&a_3+a_1^2/2&0&a_1&0&0&1&0&0&0\\
a_7+a_1a_4+a_2a_3+a_1^2a_2/2&a_4+a_1a_2&a_3+a_1^2/2&a_2&a_1&0&0&1&0&0\\
a_8+a_1a_5+a_2a_4+a_1a_2^2/2&a_5+a_2^2/2&a_4+a_1a_2&0&a_2&a_1&0&0&1&0\\
a_9+a_2a_5+a_2^3/6&0&a_5+a_2^2/2&0&0&a_2&0&0&0&1\end{array}\right).$$
The boundary of the orbit $\Gan\langle 1\rangle$ consists of all vectors with $x_0=0$ i.e. of  vectors $x\mu=x_1\mu_1+\ldots+x_9\mu_9$.
Let us calculate the product of the matrix and the element $x\mu$:
$$\rho(a)(x\mu)=x_1\mu_1+x_2\mu_2+(a_1x_1+x_3)\mu_3+(a_2x_1+a_1x_2+x_4)\mu_4+(a_2x_2+x_5)\mu_5+$$
$$+((a_3+a_1^2/2)x_1+a_1x_3+x_6)\mu_6+((a_4+a_1a_2)x_1+(a_3+a_1^2/2)x_2+a_2x_3+a_1x_4+x_7)\mu_7+$$
$$+((a_5+a_2^2/2)x_1+(a_4+a_1a_2)x_2+a_2x_4+a_1x_5+x_8)\mu_8+((a_5+a_1^2/2)x_2+a_2x_5+x_9)\mu_9.$$
Let us consider  elements that satisfy the conditions of the type $x_1=0,\ldots, x_c=0, x_{c+1}\ne 0$. For each $c$ we calculate the dimension of any orbit of this type and the dimension of the family of such orbits. If we take an element $x\mu$ with $x_1\ne 0$, then two first coordinates are fixed under our action; the third coordinate can take on any value (and this value determines $a_1$); the fourth also takes on any value and determines $a_2$; the fifth coordinate is fixed again, since $a_2$ is already known; the sixth coordinate detemines $a_3$, the seventh coordinate determines  $a_4$, the eigth coordinate determines $a_5$, and the ninth is fixed. Therefore, there are five free parameters $a_i$, so the orbit dimension equals to 5. The dimension of the family of all orbits with $x_1\ne 0$ is 1 less than the difference of an amount of free $x_i$ and of free $a_i$ which equals to $9-5-1=3$. Hence the orbits with $x_1\ne 0$ form a three-dimensional space of five-dimensional orbits.

Now  proceed to other $c$'s. One can note that for any $c$ the set of elements that satisfy our condition is $\Gan$-invariant. Hence if we take an element $x\mu$ with $x_1=0,\ldots, x_c=0, x_{c+1}\ne 0$, then any element from its orbit $\Gan x$ satisfies the conditions of the same type for some $c'\ge c.$ The similar argument lets us complete the following table:
$$\begin{tabular}{c|c|c}
\hline
$c$& orbit dimension & family dimension\\
\hline
0 & 5& 3\\
\hline
1&5& 2\\
\hline
2&2&4\\
\hline
3&2&3\\
\hline
4&2&3\\
\hline
5&1&3\\
\hline
6&0&2\\
\hline
7&0&1\\
\hline
8&0&0\\
\hline
\end{tabular}$$
It is obvious that the modality of this action equals to the maximum of family dimensions, which is 4.

Now we generalize this idea for any $k$. Let us introduce the renumeration of the basic elements by distributing them into subsets of $l_0=1$ element of zero power, $l_1=2$ elements of power 1, and so on:  $1=\mu_0=\mu_{l_0}, \mu_1=\mu_{l_1}, \mu_2=\mu_{l_1+1}, \mu_1^2=\mu_3=\mu_{l_2}, \ldots, \mu_1^b=\mu_{l_b}, \mu_1^{b-1}\mu_2=\mu_{l_b+1},\ldots, \mu_2^{b}=\mu_{l_b+b},\ldots$. 
  
The matrix $\rho(a)$ is constructed of $(i+1)\times(j+1)$-blocks $A_{ij}$. Blocks above the main diagonal are zeros, while diagonal blocks $A_{ii}$ are the identity matrices  $E_{i+1}$, a block $A_{i0}$ is a column with elements $y_{l_i},\ldots y_{l_i+i}$, and a block $A_{ij}$ with $i>j>0$ is of the form
$$A_{ij}=\left(\begin{array}{cccc}y_{l_{i-j}}&0&\ldots&0\\y_{l_{i-j}+1}&y_{l_{i-j}}&\ldots&0\\
\vdots&\vdots&\vdots&\vdots\\y_{l_{i-j}+i-j}&0&\ldots&\\0&y_{l_{i-j}+i-j}&\ldots&\vdots\\\vdots&\vdots&\vdots&\vdots\\
0&0&\ldots&y_{l_{i-j}+i-j}\end{array}\right).$$
Here by $y_p$ we denote some polynomials of variables $a_q$. It is useful to note that the polynomial $y_p$ contains the summand $a_p$ and polynomials with smaller indices do not depend on $a_p$.

It is easy to see that the modality of our action equals to the maximal (over all $c$) amount of parameters in the set of orbits of elements with $x_1=\ldots=x_c=0$, $x_{c+1}\ne 0$. One can obtain the orbit dimension of such element by killing first $c$ columns of the matrix and consider the rows of the resulting submatrix. It is easily  deduced from the desription of the matrix that the orbit dimension equals to the amount of different $y_i$'s in the corresponding submatrix. Hence the family dimension is 1 less than the difference between the dimension of the union of all orbits (which is the the amount of columns in the submatrix) and the amount of $y_i$'s. If we state the power of the monomial corresponding to the first column of the remainder matrix, then killing of one more column decreases the minuend and does not change the subtrahend, so it is sufficient to restrain ourselves to submatrices with first column numbered $l_i$. Suppose the first column is numbered $l_p$.  Then there are exactly $2+3+\ldots+(k-p+1)$ different $y_i$, so the amount of parameters can be calculated as $(p+1)+(p+2)+\ldots+(k+1)-2-3-\ldots(k-p+1)-1 =pk-p^2+p$. This term attains its minimum whereas  $p=[\frac{k+1}{2}]$. 

Now proceed to an action of the group $\Gan$ on a quadric $Q_n\subset\P^{n+1}$. We have already obtained the corresponding algebra  $R=\C[S_1,\ldots, S_n]/(S_k^2-S_l^2, S_iS_j\mid k,l=1,\ldots, n, i\ne j)$ and the quadratic form $q=2y_0y_{n+1}-\sum_{i=1}^n y_i^2.$ Let us describe the orbit-closure of the line $\langle v_0\rangle$ spanned by the cyclic vector $v_0=1\in R$. One can write down the representation $\rho$ with the matrices:
$$\rho(a)=\left(\begin{array}{ccccc}1&0&\ldots&0&0\\a_1&1&\ldots&\vdots&\vdots\\\vdots&\vdots&\ddots&\vdots&\vdots\\a_n&0&\ldots&1&0\\\frac{1}{2}\sum a_i^2&a_1&\ldots&a_n&1\end{array}\right),$$
where $a=(a_1,\ldots, a_n)\in\Gan$. The orbit of the line $\langle 1\rangle$ consists of all points $(1:a_1:\ldots:a_n:\frac{1}{2}\sum a_i^2)$. This orbit is dense in the quadric $2y_0y_{n+1}-\sum_{i=1}^n y_i^2=0$ and its complement is the union of all points $(0:b_1:\ldots:b_n:c)$ with $\sum b_i^2=0$.

\begin{prop}\label{modqu}
The modality of a locally transitive action of the group $\Gan$ on the quadric $Q_n$ equals to $n-2$ if $n\ge 2$ and is zero if $n=1$.
\end{prop}

\begin{proof}
If $n\ge 2$, then the orbit of a point $(0:b_1:\ldots:b_n:c)$ with $\sum b_i^2=0$ consists of all elements $(0:b_1:\ldots:b_n:c+b_1a_1+\ldots+b_na_n)$. It means that the orbit is determined by the set of the coordinates of its element except for the first and the last ones. These coordinates have to satisfy the equation  $\sum b_i^2=0$. Therefore these orbits form a $(n-2)$-dimensional family and fill all the supplement of the orbit $\Gan\langle 1\rangle $. If we have $n=1$, then the quadric contains only two orbits: the orbit of a cyclic vector and a fixed point $(0:0:1)$.
\end{proof}

\begin{cor} There is a finite number of orbits for a locally transitive action of a group $\Gan$ on $Q_n$ if and only if $n\le 2$.
\end{cor}
 
 For each $n$ the quadric contains the only fixed point which is $(0:\ldots:0:1)$. 
\end{section}

\begin{section}{Maximal commutative subalgebras}
\begin{prop}\label{Autcom} Let $X$ be an irreducible variety and $G\subseteq \Aut(X)$  a finite-dimensional commutative algebraic subgroup which has an open orbit on $X$. Then $G$ is a maximal commutative subgroup in $\Aut(X)$.
\end{prop}

\begin{proof} Consider an open orbit $Gx_0\subseteq X$ and suppose that an element $g\in\Aut(X)\setminus~G$ commutes with all elements of $G$. Since $X$ is irreducible, one can see that $gGx_0\cap Gx_0\ne\emptyset$, so there are elements $g_1, g_2\in G$ such that $gg_1x_0=g_2x_0$. Then the element $g'=g_2^{-1}gg_1$ fixes the point $x_0$. We know that $g_1, g_2\in G$ and $g\notin G$, hence $g'\notin G$; in particular, $g'\ne e$. Since the group is commutative, the element $g'$ has to act as the identity on the open orbit $Gx_0$, and on the whole variety $X$. There is a contradiction, so the proposition is proved.
\end{proof}

To each action of the group $\Gan$ on the space $\P^n$ one can assign a commutative unipotent subgroup in $\SL(n+1)$, and to each action on a quadric $Q_n\subset\P^{n+1}$ corresponds a commutative unipotent subgroup in $\SO(n+2)$. If the action is locally transitive, then from Proposition ~\ref{Autcom} it follows that the corresponding subgroup is maximal of commutative unipotent subgroups in $\SL(n+1)$ or $\SO(n+2)$ respectively. Let us pose the converse question: does there exist a maximal commutative unipotent subgroup of dimension $n$ in $\SL(n+1)$ (resp. $\SO(n+2)$) such that the corresponding action is not locally transitive? Paper \cite{S-T} contains the list of all maximal commutative nilpotent subalgebras in $\sl(n+1)$ for $n\le 4$. For each $n$ small enough to each maximal nilpotent subalgebra of dimension $n$  assign a locally transitive action $\Gan:\P^n$. We do not have an answer in general case for special linear algebras. In orthogonal case a counterexample can be obtained from the work \cite{H-W-Z}. Theorem ~\ref{unit} says that there is a unique (up to conjugacy) subalgebra of dimension $n$ in $\so(n+2)$ which corresponds to a locally transitive action $\Gan:Q_n$. In the paper \cite{H-W-Z} there are indicated so called free-rowed maximal commutative nilpotent subalgebras of dimension $n$ in $\so_{n+2}$, where $n\ge 6$. Let us describe these subalgebras and show that  corresponding actions are not locally transitive. Suppose the form $B$ is determined by the matrix
$$\left(\begin{array}{ccc}0&0&E_4\\0&E_{n-6}&0\\E_4&0&0\end{array}\right),$$
and the subalgebra consists of all the matrices 
$$\left(\begin{array}{ccc}0&A&Y\\0&0&B\\0&0&0\end{array}\right),$$
where a matrix  $Y\in\Mat_{4\times 4}$ is antisymmetric, in a matrix $A\in\Mat_{4\times (n-6)}$ all rows but the first one $(m_1,\ldots, m_{n-6})$ are zeros, and in a matrix  $B\in\Mat_{(n-6)\times 4}$ all columns but the first one $(-m_1,\ldots, -m_{n-6})$ are zeros.
It is easy to check that it is a subalgebra in $\so_{n+2}$ and its dimension equals $n$. If $X$ is a matrix of such type, then its square equals to 
$(-m_1^2-\ldots-m_{n-6}^2) E_{1,n-1}$, and its cube is zero, so its exponent is $E+X+\frac{1}{2}(-m_1^2-\ldots-m_{n-6}^2)E_{1, n-1}$. If we multiply any column by this exponent, then its four last coordinates will not change, therefore there is no open orbit for this action.   
\end{section}

\address {Chair of Higher Algebra, Faculty of Mechanics and Mathematics, Moscow
State University, Leninskie Gory, GSP-1, Moscow, 119991, Russia }

\email {e-mail: sharojko@mccme.ru}

\end{document}